\documentclass[12pt]{article}
\usepackage{amsfonts,amsthm,latexsym,amsmath,amssymb,amscd,epsfig,psfrag,enumerate}
\usepackage{graphics,graphicx, bezier, float, color}
\usepackage{graphicx}
\usepackage{caption}

\newtheorem{theo+}           {Theorem}
\newtheorem{prop+}           {Proposition}
\newtheorem{coro+}           {Corollary}
\newtheorem{lemm+}           {Lemma}

\newtheorem{ques+}{Question}
\theoremstyle{definition}

\newtheorem{not+}            {Notation}

\newtheorem{Ex}              {Example}[section]

\newtheorem*{ack}            {Acknowledgements}

\newtheorem{rema+}           {Remark}
\newenvironment{theorem}{\begin{theo+}}{\end{theo+}}
\newenvironment{proposition}{\begin{prop+}}{\end{prop+}}
\newenvironment{corollary}{\begin{coro+}}{\end{coro+}}
\newenvironment{lemma}{\begin{lemm+}}{\end{lemm+}}

\newcommand{\bq}{\begin{eqnarray}}
\newcommand{\eq}{\end{eqnarray}}
\newcommand{\beq}{\begin{eqnarray*}}
\newcommand{\eeq}{\end{eqnarray*}}

\newtheorem{definition}{\bf Definition}[section]

\newtheorem{claim}{\bf Claim}[section]

\date{}

\begin{document}

\thispagestyle{empty}
\title{Morphic elements in regular near-rings}
\author{Alex Samuel Bamunoba, Ivan Philly Kimuli\footnote{Corresponding author}~   and David Ssevviiri}
\maketitle
\begin{center}
Department of Mathematics, Makerere University,\\
P.O. Box 7062, Kampala Uganda\\
E-mail addresses:  bamunoba@cns.mak.ac.ug; kpikimuli@gmail.com;  david.ssevviiri@mak.ac.ug
\end{center}
\begin{abstract}
\noindent We define morphic near-ring elements and study their behavior in regular near-rings. We show that the class of left morphic regular near-rings is properly contained between the classes of left strongly regular and unit regular near-rings.
\end{abstract}

\textbf{Keywords}: Morphic regular, strongly regular,  unit regular

\textbf{Mathematics Subject Classification (2020)}: 16Y30, 16E50

\maketitle



\maketitle
\section{Introduction}

Let $R$ be a ring, $M$ be an $R$-module and $A:=\text{End}_{R}(M)$ be the ring of $R$-module endomorphisms of $M$.  An element $\alpha$ of $A$ is called {\it regular} if there exists a $\beta\in A$ such that $\alpha \beta\alpha=\alpha$.     A regular element $\alpha$ is called {\it unit-regular} if $\beta$ can be chosen to be an $R$-automorphism of $M$, i.e., $\beta\in {\rm Aut}_R(M)$.  If every element in $A$ is unit-regular, then $A$ is a unit-regular ring.
In two successive papers, \cite{ehrlich1968unit} and \cite{ehrlich1976units}, G. Ehrlich studied the class of unit-regular rings.
She  showed that $\alpha\in A$ is unit-regular if and only if $\alpha$ is regular and has the property that $M/\text{Im}(\alpha) \cong \ker(\alpha)$.  Note that, the property  $M/\text{Im}(\alpha) \cong \ker(\alpha)$  is  the {\it Dual of the First Isomorphism Theorem}, $M/\ker(\alpha)\cong \text{Im}(\alpha)$, for $\alpha \in A$.  Ehrlich's result sparked further studies concerning this dual.\\

\noindent Let $a\in R$.  Following the case for $M ={}_{R}R$ in which $\alpha(r)=ra$ for all $r\in R$, we denote the  left and right annihilators of $a$  by $(0:_la)$ and $(0:_ra)$ respectively.  It is known that for any  $a\in R,~ R/(0:_l a)\cong Ra$.  By dualisiling this isomorphism, W. K. Nicholson and C. E. S$\acute{\text{a}}$nchez  in \cite{nicholson2004rings}  introduced the notion of morphic elements. An element $a$ in $R$ is said to be {\it left morphic} if $R/Ra\cong (0:_l a)$.  A ring is called  left morphic if each of its elements is left morphic, i.e., if  for every $a\in R$, the Dual of the First Isomorphism Theorem for an endomorphism $r\mapsto ra$ of $R$  holds.  Right morphic elements and right morphic rings are defined analogously.  We say that $R$ is morphic if it is both left and right morphic.  With  appropriate modifications, morphic modules, morphic groups and morphic group-rings were  defined.  These studies exist in \cite{chen, lee2007morphic, lee2010theorem,  li2010morphic,  nicholson2005morphic} among others.  In this work we extend this study to near-rings.\\

\noindent Let $N$ be a near-ring.  By defining left morphic near-ring elements, several properties   are obtained.
For instance, we  establish that where as a right near-ring is in general not Abelian as a group under addition and also
 lacks left-distributivity of multiplication over addition; existence of a left morphic idempotent element in a zero-symmetric right near-ring allows the two mentioned conditions above to hold under some special   situations.  In particular, if $e$  is a left morphic idempotent  element of a zero-symmetric near-ring $N$, then for all $x\in N$,
$$x(1-e)=x-xe=-xe+x.$$
Secondly, we establish in Theorem~\ref{217} that an idempotent $e$ of a  zero-symmetric
near-ring $N$ is left morphic if and only if   $e$ and its supplement, $1-e$, are  orthogonal   and their left annihilator sets have only the zero element in common.  Moreover, every left morphic idempotent element of a  near-ring is morphic.
Surprisingly, every idempotent need not be left morphic.   It turns out that there are many results already known in left morphic rings that do not carry over to left morphic near-rings.  For instance, for rings, Theorem~\ref{T}  was proved.

\begin{theorem}{\rm \cite[Ehrlich]{ehrlich1976units}}\label{T}
A ring $R$  is unit regular if and only if it is both regular and left morphic.
\end{theorem}

\noindent The group-theoretic analogue of Theorem \ref{T} was proved by Y. Li and W. K. Nicholson in \cite{Li}.
Here we  show that the analogue of Theorem \ref{T} does not hold for near-rings.
Whereas a left morphic regular zero-symmetric near-ring is unit-regular, see Theorem~\ref{62}, we show in Examples~\ref{20}  that the converse does not hold in general.   Also, whereas zero-symmetric left strongly regular near-rings are left morphic, see Proposition~\ref{ff}, the converse is not true in general as shown in Example~\ref{gggg}.
We establish and classify  left morphic regular near-rings.  The relationship between left strongly regular, left morphic regular and unit-regular near-rings is studied.  It is explicitly illustrated that the class of left morphic regular near-rings properly lies between the classes of left strongly regular and unit-regular near-rings.

\section{ Preliminaries on right near-rings}
A   {\it right  near-ring} $(N,+,\star)$ is a non-empty set $N$ with two binary operations `$+$' and `$\star$' such that $(N,+)$ is a group, (not necessarily Abelian) with identity $0,~ (N,\star)$ is a semigroup and `$\star$' satisfies the right distributive
law with respect to `$+$', i.e., $(x + y)\star z = x\star z + y\star z$ for all $x, y, z \in N$.   One often writes $xy$  for $x\star y$.
A natural example of a right near-ring is the following:  Let $G$ be a group, written additively but not necessarily Abelian, and  $M(G)$  be the set of all mappings from $G$ to $G$.   Given $f,g$ in $M(G)$, the mapping $f+g$ from $G$ to $G$ is defined as $(f+g)(x)=f(x)+g(x)$.    The product $f\star g$, of two elements, is their composition $f(g(x))$,  for all $x$ in $G$.  Then $(M(G),+,\star)$ is a right near-ring.
A {\it near-field} is a near-ring in which there is a multiplicative identity and every non-zero element has a multiplicative inverse.  As an immediate consequence of the right distributive law, we have $0x=0$ for every $x$ in $N$ but it is not necessarily true that $x0=0$ for every $x$ in $N$.  However, if $x0=0$ for every $x\in N$, then $N$ is zero-symmetric.\\

\noindent Let $a\in N$.  For a right near-ring $N$,   the mapping $r \mapsto ra$ from $N$ to $N$
determines an endomorphism of $(N, +)$ whose kernel coincides with  $(0:_la)$ and its image with $Na$.

\begin{definition}Let $N$ be a right near-ring.
{A left $N$-\emph{module} $M$ is a group, $(M, +)$,  written additively but not necessarily Abelian together with a map $N \times M \rightarrow M,(r, m) \mapsto rm,$
such that for all $m\in M~\text{and for all}~ r_1, r_2\in N$,   $$(r_1 + r_2)m = r_1m + r_2m~~\text{and}~~  (r_1r_2)m = r_1(r_2m).$$
}
\end{definition}

\noindent Observe that, a right near-ring $N$ has a natural structure of a left $N$-module.

\begin{definition}~\label{xxx}
Let $N$ be a right near-ring and $M$ be a left $N$-module. A subset $L$ of $M$ is called an
 {\it $N$-ideal} of $M$ if $(L,+)$ is a normal subgroup of $(M,+)$ and  $$r(l+m)-rm\in L$$ for all $m\in M$, $l\in L$ and $r\in N$.  The $N$-ideals of ${}_{N}N$ are called {\it left ideals}.  In fact, if $L$ is a left ideal of $N$ and $LN\subseteq L$, then $L$ is an {\it ideal} of $N$.\\
\end{definition}

\noindent For any near-ring $N$ and $a\in N$, the set $(0:_la)$ is an  $N$-ideal.
If $M$ is a left $N$-module, one can define a factor left $N$-module $M/L$ provided $L$ is an $N$-ideal of $M$. Near-ring homomorphisms and $N$-homomorphisms
are defined in the usual manner as for rings.\\

\noindent Let $a\in N$. We call $a$ {\it regular} if  there exists  $x\in N$ such that $a = axa$.  If $x$ can be chosen to be a unit, then $a$ is called {\it unit-regular}.   The element $a$ is said to be {\it left (right) strongly regular} if  there exists $x \in N$ such that $a = xa^2~(a = a^2x)$ and  {\it strongly regular} if it is left and right strongly regular.  Regular, unit-regular and left (right) strongly regular near-rings are defined in the usual manner.  We say that $N$ is {\it reduced} if it has no non-zero nilpotent elements, i.e., $a \in N, \ a^n=0$  implies $a=0$.  $N$ is said to have IFP ({\it Insertion-of-Factors-Property}) if for $a, b \in N,~ ab = 0$ implies $anb = 0$ for every $n \in N$.  For a detailed account of basic concepts concerning  near-rings and near-ring modules, we refer the reader to the books
\cite{ferrero2012nearrings}, \cite{meldrum1985near} and \cite{Pilz:Nearrings}.\\

\noindent\textbf{Convention}:     A near-ring $N$ is called {\it unital} if there exists an element $1\neq 0$ in $N$ such that, $1r=r1=r$ for every element $r$ in $N$.  We call  $N$ \emph{zero-symmetric} if $r0=0r=0$ for every $r$ in $N$.
Unless otherwise stated, all near-rings are  unital  zero-symmetric right near-rings and all $N$-modules are unitary.
We denote by $U(N)$ and $U(R)$  the collection of all units of $N$ and $R$ respectively (i.e., invertible elements with respect to multiplication `$\star$').

\section{Preliminaries on right near-rings}

A   {\it right  near-ring} $(N,+,\star)$ is a non-empty set $N$ with two binary operations `$+$' and `$\star$' such that $(N,+)$ is a group, (not necessarily Abelian) with identity $0,~ (N,\star)$ is a semigroup and `$\star$' satisfies the right distributive
law with respect to `$+$', i.e., $(x + y)\star z = x\star z + y\star z$ for all $x, y, z \in N$.   One often writes $xy$  for $x\star y$.
A natural example of a right near-ring is the following:  Let $G$ be a group, written additively but not necessarily Abelian, and  $M(G)$  be the set of all mappings from $G$ to $G$.   Given $f,g$ in $M(G)$, the mapping $f+g$ from $G$ to $G$ is defined as $(f+g)(x)=f(x)+g(x)$.    The product $f\star g$, of two elements, is their composition $f(g(x))$,  for all $x$ in $G$.  Then $(M(G),+,\star)$ is a right near-ring.
A {\it near-field} is a near-ring in which there is a multiplicative identity and every non-zero element has a multiplicative inverse.  As an immediate consequence of the right distributive law, we have $0x=0$ for every $x$ in $N$ but it is not necessarily true that $x0=0$ for every $x$ in $N$.  However, if $x0=0$ for every $x\in N$, then $N$ is zero-symmetric.\\

\noindent Let $a\in N$.  For a right near-ring $N$,   the mapping $r \mapsto ra$ from $N$ to $N$
determines an endomorphism of $(N, +)$ whose kernel coincides with  $(0:_la)$ and its image with $Na$.

\begin{definition}Let $N$ be a right near-ring.
{A left $N$-\emph{module} $M$ is a group, $(M, +)$,  written additively but not necessarily Abelian together with a map $N \times M \rightarrow M,(r, m) \mapsto rm,$
such that for all $m\in M~\text{and for all}~ r_1, r_2\in N$,   $$(r_1 + r_2)m = r_1m + r_2m~~\text{and}~~  (r_1r_2)m = r_1(r_2m).$$
}
\end{definition}

\noindent Observe that, a right near-ring $N$ has a natural structure of a left $N$-module.

\begin{definition}~\label{xxx}
Let $N$ be a right near-ring and $M$ be a left $N$-module. A subset $L$ of $M$ is called an
 {\it $N$-ideal} of $M$ if $(L,+)$ is a normal subgroup of $(M,+)$ and  $$r(l+m)-rm\in L$$ for all $m\in M$, $l\in L$ and $r\in N$.  The $N$-ideals of ${}_{N}N$ are called {\it left ideals}.  In fact, if $L$ is a left ideal of $N$ and $LN\subseteq L$, then $L$ is an {\it ideal} of $N$.\\
\end{definition}

\noindent For any near-ring $N$ and $a\in N$, the set $(0:_la)$ is an  $N$-ideal.
If $M$ is a left $N$-module, one can define a factor left $N$-module $M/L$ provided $L$ is an $N$-ideal of $M$. Near-ring homomorphisms and $N$-homomorphisms
are defined in the usual manner as for rings.\\

\noindent Let $a\in N$. We call $a$ {\it regular} if  there exists  $x\in N$ such that $a = axa$.  If $x$ can be chosen to be a unit, then $a$ is called {\it unit-regular}.   The element $a$ is said to be {\it left (right) strongly regular} if  there exists $x \in N$ such that $a = xa^2~(a = a^2x)$ and  {\it strongly regular} if it is left and right strongly regular.  Regular, unit-regular and left (right) strongly regular near-rings are defined in the usual manner.  We say that $N$ is {\it reduced} if it has no non-zero nilpotent elements, i.e., $a \in N, \ a^n=0$  implies $a=0$.  $N$ is said to have IFP ({\it Insertion-of-Factors-Property}) if for $a, b \in N,~ ab = 0$ implies $anb = 0$ for every $n \in N$.  For a detailed account of basic concepts concerning  near-rings and near-ring modules, we refer the reader to the books
\cite{ferrero2012nearrings}, \cite{meldrum1985near} and \cite{Pilz:Nearrings}.\\

\noindent\textbf{Convention}:     A near-ring $N$ is called {\it unital} if there exists an element $1\neq 0$ in $N$ such that, $1r=r1=r$ for every element $r$ in $N$.  We call  $N$ \emph{zero-symmetric} if $r0=0r=0$ for every $r$ in $N$.
Unless otherwise stated, all near-rings are  unital  zero-symmetric right near-rings and all $N$-modules are unitary.
We denote by $U(N)$ and $U(R)$  the collection of all units of $N$ and $R$ respectively (i.e., invertible elements with respect to multiplication `$\star$').

\section{General results and examples of left morphic near-rings}

\noindent For a near-ring $N$ and $a\in N$, consider the multiplication map  $n\mapsto na: N\rightarrow N.$
By the First Isomorphism Theorem, $ N/(0:_la)\cong~Na$  considered as $N$-modules.  If $Na$ is an $N$-ideal, then   the dual to this  isomorphism  is
$N/Na\cong(0:_la).$

\begin{definition}~\label{204}
Let $N$ be a near-ring and $a\in N$.  An element $a$ is called \emph{left morphic} if $Na$ is an $N$-ideal and
\[  N/Na\cong(0:_la)\] as  near-ring $N$-modules.  A near-ring $N$ is called \emph{left morphic} if each of its elements is left morphic.
\end{definition}

\noindent In Lemma~\ref{1}, below, we give other equivalent statements for a left morphic near-ring element.

\begin{lemma}\label{1}
Let $N$ be a near-ring and $a\in N$.  The following  are equivalent:
\begin{enumerate}
\item[\emph{(1)}] $a$ is left morphic.
\item[\emph{(2)}] There exists $b \in N$ such that $Na = (0:_lb)$ and $Nb=(0:_la)$.
\item[\emph{(3)}] There exists $b\in N$ such that $Na = (0:_lb)$ and $ Nb\cong(0:_la)$.
\end{enumerate}
\end{lemma}

\begin{proof}
$(1)\Rightarrow(2)$. By (1),  if $\sigma:N/Na \rightarrow (0:_la)$ defined by  $\sigma(1+Na)=b$ is the $N$-isomorphism, then  $Nb=\text{Im}(\sigma)=(0:_la)$ because $\sigma$ is surjective and $(0:_lb)=\ker(\sigma)=\{n+Na:nb=0\}=Na$ because $\sigma$ is injective.

$(2)\Rightarrow (3)$. This follows from the equality $Nb=(0:_la)$ of $N$-ideals.

$(3)\Rightarrow (1)$. Suppose that there exists $b \in N$ such that $Na=(0:_lb)~\text{and}~Nb \cong(0:_la)$, then $Na$ is an $N$-ideal and $N/Na =N/(0:_lb)\cong Nb \cong (0:_la)$, proving $(1)$.
\end{proof}

\begin{lemma}\label{10}
 Let $N$ be a   near-ring, $a\in N$ and $u\in U(N)$, then
 \begin{enumerate}
  \item [\emph{(1)}] $Nu=N$,
  \item [\emph{(2)}] $(0:_la)=(0:_lau^{-1})$,
  \item [\emph{(3)}]$(0:_la)u^{-1}=(0:_lua)$,
  \item [\emph{(4)}]$(0:_la)\cong(0:_la)u$.
 \end{enumerate}
\end{lemma}

\begin{proof}
To prove (1), it is enough to show that $N\subseteq Nu$ since for any $u\in U(N)$,  $1=u^{-1}u \in Nu$.   For every $r\in N, r = r\cdot1=r(u^{-1}u)=(ru^{-1})u \in Nu$, hence $N = Nu$.\\
For (2),  let $x\in (0:_la)$.   Then $xa=0$.  Then  $0=0u^{-1}=(xa)u^{-1}=x(au^{-1})$; hence $x\in (0:_lau^{-1})$ and $(0:_la) \subseteq (0:_lau^{-1})$.   For $x\in (0:_lau^{-1})$, we get  $x(au^{-1})=(xa)u^{-1}=0$.   This implies that $(xa)u^{-1}u=xa=0$.  Therefore, $x\in(0:_la)$; which proves $(0:_lau^{-1})= (0:_la)$.   The proof of (3) is similar to that of (2).\\
For the proof of (4), the  map $\vartheta:(0:_la)\rightarrow (0:_la)u,~x\mapsto xu$  gives the required $N$-isomorphism.
\end{proof}

\begin{proposition}\label{2}
Let $N$ be a near-ring and $a\in N$.  If $a$ is  left morphic, then the same is true of $au$ and $ua$, for every $u \in U(N)$.
\end{proposition}
\begin{proof} Let $a \in N$ be a left morphic element.  Then, by Lemma~\ref{1}, $Na = (0:_lb)~\text{and}~Nb=(0:_la)$ for some $b \in N$.  Suppose $u \in N$  is a unit.  Applying Lemmas~\ref{1} and \ref{10}, $N(ua)=Na=(0:_lb)=(0:_lbu^{-1})$ and  $N(bu^{-1})=(Nb)u^{-1}=(0:_la)u^{-1}=(0:_lua)$.  So $ua$ is left morphic.
Similarly, $N(au)=(Na)u=(0:_lb)u=(0:_lu^{-1}b)$,  and $N(u^{-1}b)=(Nu^{-1})b=Nb=(0:_la)=(0:_lau)$.  So, $au$ is also left morphic.
\end{proof}

\begin{proposition}\label{64}
Let $N$ be a near-ring and $a\in N$.  If $a\in N$ is  left morphic, then the following are equivalent:
\begin{enumerate}
\item[\emph{(1)}]$(0:_la)= \{0\}$,
\item[\emph{(2)}]$Na = N$,
\item[\emph{(3)}]$a \in U(N)$.
\end{enumerate}
\end{proposition}
\begin{proof}
$(1)\Leftrightarrow(2)$ is clear from $N/Na$ isomorphic to $(0:_la)$. The proof of $(3)\Rightarrow(1)$ is immediate. We now prove $(2)\Rightarrow(3)$;  from (2) we have that $1 = ba$  for some $b$ in $N$ and so $a=aba$.  Then $(1-ab)$ is in the left annihilator of $a$ that is zero.   Thus we obtain  $1=ab=ba$; which implies that $a\in U(N)$.
\end{proof}

\noindent Using condition (2) in Lemma~\ref{1} we get:
A finite direct product $\prod^{n}\limits_{i=1} N_{i}$ of near-rings is left morphic if and only if each  near-ring $N_i$ is left morphic.
The proof is immediate after taking into account that the left annihilator of an element $(a_1,a_2,\ldots,a_n)$ in the product is just the product of the sets $(0:_la_i),i=1,2,\ldots,n$.\\

\noindent We will give examples of left morphic near-ring elements and left morphic near-rings starting with the ring theoretic ones.

\begin{Ex}~\label{der}
The earliest known class of left morphic rings to be determined were  the unit-regular rings in \cite{ehrlich1976units}.  Later, in \cite{nicholson2004rings}, \cite{nicholson2004principal} and \cite{nicholson2005morphic},  more examples were investigated to include  all local rings where the  Jacobson radical is a nilpotent principle ideal,  all factor rings of commutative principal ideal domains like the rings $\mathbb{Z}_n$ of integers modulo $n$.  Let $R$ be a unit-regular ring, $\sigma:R\to R$ be any ring homomorphism that fixes idempotents element-wise and $R[x;\sigma]$ be the skew polynomials with commutation relation $xr=\sigma(r)x$.  The rings  $R[x;\sigma]/(x^{n})$ and $R[x]/(x^{n})$ are left morphic  for all $n\geq 1$ by \cite{lee2007morphic} and \cite{lee2010theorem}.
\end{Ex}

\begin{Ex}\label{hth}
Let $N$ be a near-ring.  The units of $N$ are left morphic.
To see this, let $u\in N$ be a unit.   Then $Nu=N$ by Lemma~\ref{1} and $N/Nu =N/N \cong \{0\}=(0:_lu)$.  This proves that $u$ is left morphic.  Similarly, $N/uN\cong (0:_ru)$ and thus $u$ is morphic.   Consequently,  all  near-fields are morphic.
If all idempotents $e$ of  $N$ are  central,  then     $1-e$  is  idempotent and   $N(1-e) \subseteq(0:_le)$.   Let $xe=0$.  Then $x-(1-e)x=x+ex-x=0$.  We, therefore, have $x=(1-e)x=x(1-e)\in N(1-e)$.   Thus $N(1-e) =(0:_le)$.  Similarly,  $(0:_l(1-e))=Ne$  and  Lemma~\ref{1} implies  $e$ is  left morphic.
\end{Ex}

\begin{Ex}~\label{ccc}
 A near-ring $N$ is  said to be  {\it subcommutative} if $Na=aN$ for all $a\in N$.   We call $N$   a {\it generalised near field} if  it is  regular  and  subcommutative.
It was proved by {\rm\cite[Theorem 1.5]{nandakumar2009weakly}} that, if $N$ is a unital generalised near field, then  for all $a\in N,Na$  is an $N$-ideal
of $N$ and $N$ has a decomposition $N= (0:_la)\bigoplus Na$.    Since it is a direct sum decomposition and both $(0:_la)$ and $Na$ are ideals of $N$,  $N/Na\cong (0:_la)$ for all $a\in N$.   Thus $N$ is left-morphic,   and hence all generalised near fields are left-morphic near-rings.
\end{Ex}

\begin{Ex}~\label{wsw} A near-ring $(N,+,\star)$ is called {\it weakly divisible} if and only if for all $a,b\in N$, there exists $x\in N$ such that $xa=b$ or $xb = a$.
Let $N$ be a finite  weakly divisible near-ring and $L = N \setminus U(N)$.  In {\rm\cite{benini1999weakly, MR2597325}}, it was shown that there exists $n,k\in \{0,1,2\ldots\},0\leq k< n$ and  $r\in L$ such that $L^k=Nr^k$, $L^n = 0$ and $r^n=0$.  Further, for each $a\in L,~a=ur^k$ for some $u\in U(N)$.     Moreover, every $L^k$ is an $N$-ideal of $N$ and  $(0:_lr^k) = Nr^{n-k}$.
\begin{claim}
The  near-ring in Example~\ref{wsw} is  left-morphic.
\end{claim}
\begin{proof}
 To see this, suppose that $a\in N$.  If $a\in U(N)$, then by Example~\ref{hth} (a) there is nothing to prove.  Otherwise $a=ur^k$ for some $k>0$.  Applying  Lemma~\ref{10},  $Na=Nur^k=Nr^k\subseteq (0:_lr^{n-k})$.   To prove $Nr^k=(0:_lr^{n-k})$, let  $x\in (0:_lr^{n-k})$.   Then  $x\in L=Nr^k$ or $x$ is a unit.   If $x$ is a unit, then $r^{n-k}=0$ and thus $N=(0:_lr^{n-k})$.   But this would imply that $n-k\geq n$  giving $k=0$, a contradiction.  Thus $x\in L=Nr^k$ and $Nr^k=(0:_lr^{n-k})$.   Since $Na=(0:_lr^{n-k})$ and (in view of   Lemma~\ref{10}) $Nr^{n-k}=(0:_lr^k)=(0:_lur^k)u=(0:_la)u\cong (0:_la),$
  $a$ is left-morphic by Lemma~\ref{1}, and thus $N$ is a left-morphic near-ring.
\end{proof}
\end{Ex}

\noindent With Example~\ref{der} in mind and following \cite[Proposition 3]{benini1999weakly}, we hasten to note that the only morphic weakly divisible near-rings of the form $N:=(\mathbb{Z}_n,+,\star)$  are those for which  $n$ is a power of a prime number.

\section{Morphic regular elements}

\noindent The aim of this section is to investigate morphic elements in regular near-rings.  The following Lemmas [3 - 6] will be useful in the sequel.

\begin{lemma}{\rm \cite[Theorem 9.158]{Pilz:Nearrings}}~\label{213}
Let $N$ be a non-zero regular near-ring.  The  following  are equivalent:
\begin{enumerate}
\item [\emph{(1)}]$N$ is reduced,
\item [\emph{(2)}]all idempotents of $N$ are central,
\item [\emph{(3)}]$N$ is a sub-direct product of near-fields.
\end{enumerate}
\end{lemma}

\begin{lemma}\label{hdt}
A near-ring $N$ is left strongly regular if and only if $N$ is regular and reduced if and only if $N$ is regular with  central idempotents
\end{lemma}
\begin{proof}
By \cite[Corollary 4.3]{Groenewald}, $N$ is left strongly regular if and only if $N$ is regular  and reduced.  Lemma~\ref{213} completes the proof.
\end{proof}

\begin{lemma}{\rm \cite[Theorem 3]{Reddy}}\label{13}
Let $N$ be a left strongly regular near-ring and $a\in N$.   If $a = xa^2$ for some $x \in N$, then $a=axa$ and $ax = xa$.
 \end{lemma}

\begin{lemma}~{\rm\cite[Proposition 5]{gord}}~\label{ffff}
Every left strongly regular near-ring is unit-regular.
\end{lemma}

\begin{proposition}~\label{ff}
Let $N$ be a left strongly regular near-ring and $a\in N$. Then $a^2$ is a regular element of $N$.
\end{proposition}
\begin{proof}
Let  $a\in N$.  Then $a=xa^2$ for some $x\in N$.  By Lemmas~\ref{hdt} and~\ref{13},   $a=axa$, and $ax=xa$ is a central idempotent.  Hence $a^2=aa=(axa)(axa)=(aax)(xaa)=a^2(xx)a^2=a^2x^2a^2$.
\end{proof}

\begin{proposition}~\label{ff}
Every  left strongly regular  near-ring $N$ is left morphic.
\end{proposition}
\begin{proof}
Let $a\in N$.   By Lemma~\ref{ffff},  $a=aua$ with $u\in U(N)$.  Using Lemma~\ref{hdt} and Lemma~\ref{13}, $au=ua$ is a central idempotent.
 Let $v$ be the inverse of $u$.  Then $a=a(uv)=(au)v$ is a product of a central idempotent and a unit which are both left morphic (according to Example~\ref{hth}).   Proposition~\ref{2} completes the proof.
\end{proof}

\begin{corollary}~\label{fff}
Every  strongly regular near-ring $N$ is left morphic.
\end{corollary}

\begin{proof}
Strongly regular near-rings are both left and right strongly regular.   By Proposition~\ref{ff}, they are  left  morphic.
\end{proof}

\noindent H. E. Heatherly and J. R. Courville in  \cite{heatherly1999near} initiated the study of near-rings with  a special condition on idempotents.
In Theorem \ref{217},  we extend their ideas and show that the notion of left morphic idempotent elements for near-rings had been already studied  in  \cite{heatherly1999near}  although was not called so at that time.
In preparation for this result, we first  prove Lemma~\ref{this}.

\begin{lemma}~\label{this}
Let $\omega$ be an idempotent and left-morphic element of the near-ring $N$.   Then $1-(1-\omega)=\omega,1-\omega$ is idempotent, $N\omega=(0:_l1-\omega),N(1-\omega)=(0:_l\omega)$ and $1-\omega$ is left-morphic.
\end{lemma}
\begin{proof}
Since $\omega$ is left-morphic, $N\omega$ is an $N$-ideal of $N$ which contains $-\omega=(-1)\omega$ and so $\omega(1-\omega)-1\omega\in N\omega$.  Thus $\omega(1-\omega)-\omega=n\omega$ for some $n\in N$.  Right multiplication by $\omega$ gives $-\omega=n\omega$ and hence $\omega(1-\omega)=0$.  This proves $N\omega\subseteq(0:_l1-\omega)$.  Let $x\in (0:_l1-\omega)$.  Then $x(1-\omega)=0$ and $-x=x(1-\omega)-x\in N\omega$ since $N\omega$ is an $N$-ideal of $N$.  Thus $x\in N\omega$ and $N\omega=(0:_l1-\omega)$ follows.  Further, $\omega(1-\omega)=0$ implies $1-\omega$ is idempotent.  By \cite{heatherly1999near}, we then have $1-(1-\omega)=\omega$ and $N(1-\omega)=(0:_l\omega)$.  Thus $1-\omega$ is left-morphic by Lemma~\ref{1}.
\end{proof}
\begin{theorem}~\label{217}
Let $N$ be a  near-ring and  $e$  be an idempotent  element of $N$.  The following statements  are equivalent:
\begin{enumerate}
\item [\emph{(1)}] $e$ is  left-morphic,
\item [\emph{(2)}] $Ne=(0:_l1-e)$,
\item [\emph{(3)}] $x(1 - e) = -xe + x$ for all $x\in N$,
\item [\emph{(4)}] $(0:_le)\cap (0:_l1-e)=0$ and $ e(1- e) =0$,
\item [\emph{(5)}] $x(1-e)=x-xe$ for all $x\in N$,
\item [\emph{(6)}] $N(1-e)=(0:_le)$ and $e(1- e) =0$,
\item [\emph{(7)}] $1-e$ is  left-morphic and idempotent.
\end{enumerate}
\end{theorem}
\begin{proof}~
Since the implications $(2)\Leftrightarrow (3)\Leftrightarrow (4)\Leftrightarrow (5)$ are by \cite[Proposition 2.1]{heatherly1999near},  we prove $(2)\Rightarrow(1)\Rightarrow(2)\Rightarrow(6)\Rightarrow(2)$ and $(1)\Rightarrow(7)\Rightarrow(1)$.
\begin{itemize}
  \item [(2)$\Rightarrow$(1).]Suppose that (2) holds; it suffices to show that $N(1-e)=(0:_le)$ by Lemma~\ref{1}.   In view of  \cite[Corollary 2.2]{heatherly1999near}, $1-e$ is an idempotent
      element of $N,1-(1-e)=e$ and $N(1-e)=(0:_l(1-(1-e))).$  Thus $N(1-e) =(0:_le)$.   Therefore, applying Lemma \ref{1}, $Ne=(0:_l1-e)$ and $N(1-e)=(0:_le)$ implies that $e$ is left-morphic.
 \item [(1)$\Rightarrow$(2)]  and (1)$\Rightarrow$(7)  follow immediately from Lemma~\ref{this}  with $\omega=e$.

\item [(7)$\Rightarrow$(1).]  This follows by Lemma~\ref{this} with $\omega=1-e$.

\item [(2)$\Rightarrow$(6).] Following  \cite[Corollary 2.2]{heatherly1999near}, $1-e$ is an idempotent element of $N$ and $1-(1-e)=e$.    Thus $e(1-e)=0$ and   $N(1-e)=(0:_l(1-(1-e)))=(0:_le)$.

\item [(6)$\Rightarrow$(2).]  Since the property  $e(1- e) =0$ gives  $(1-e)$ idempotent, using Lemma~\ref{this} with $\omega=1-e$ yields (2).
\end{itemize}
\end{proof}

\begin{Ex}~\label{cccxx}
Consider the commutative ring $N$  defined on the Klein four group $(N,+)$ with $N = \{0, a, b, c\}$,
whose addition   and  multiplication is given by the two tables:
\begin{center}
\begin{minipage}[t]{0.3\linewidth}

\begin{tabular}{c||cccc}\label{ssss}
     $+$ &    $0$ & $a$ & $b$ & $c$ \\\hline\hline
     $0$ &    $0$ & $a$ & $b$ & $c$ \\
     $a$ &    $a$ & $0$ & $c$ & $b$ \\
     $b$ &    $b$ & $c$ & $0$ & $a$ \\
     $c$ &    $c$ & $b$ & $a$ & $0$ \\
     \hline\hline
   \end{tabular}
\end{minipage}
\begin{minipage}[t]{0.3\linewidth}

\begin{tabular}{c||cccc}
     $\star$& $0$ & $a$ & $b$ & $c$ \\\hline\hline
     $0$ &    $0$ & $0$ & $0$ & $0$ \\
     $a$ &    $0$ & $a$ & $a$ & $0$ \\
     $b$ &    $0$ & $a$ & $b$ & $c$ \\
     $c$ &    $0$ & $0$ & $c$ & $c$ \\
     \hline\hline
   \end{tabular}
\end{minipage}
\end{center}
We can verify that $N$  is left morphic  since $Na=(0:_lc)$, $Nc=(0:_la)$,  $(0:_lb)=\{0\}$ and $(0:_l0)=Nb$.
\end{Ex}

\noindent Near-rings in which every element is idempotent are called \emph{Boolean}.   In light of Example~\ref{cccxx}  and a property exhibited by a left morphic idempotent element of a  near-ring in Theorem~\ref{217}, one is    led to  the following:

\begin{proposition}~\label{cccxi}
Let $N$ be a Boolean near-ring.  Then $N$ is a commutative morphic ring.
\end{proposition}

\begin{proof}
By \cite[Theorem 3]{bell1970near}, $(N,+)$ is commutative.  $N$ is (left and right) strongly regular and so is morphic by Corollary~\ref{fff}.  By Lemma~\ref{hdt}, all the elements of $N$ are central.  Now $$x(y+z)=(y+z)x=yx+zx=xy+xz$$ and we have the left distributive law.
\end{proof}

\noindent A near-ring $N$ is called a \emph{left duo near-ring} if the left ideals of $N$ are ideals of $N$.  Every left strongly regular near-ring is a  left duo near-ring.

\begin{Ex}~\label{gggg}
Let  $D$  be a division ring    and  $R:=M_n(D)$ be the $n\times n$ matrix ring over $D,n\geq 2$.   It is well known  that  $R$  is a left and right  morphic regular ring, see~\cite{nicholson2004rings}.  However, $R$ is neither left nor right duo, hence neither left nor right strongly regular.  In particular,
$$\left[\begin{array}{cc}
0 & 1 \\
0 & 0 \\
\end{array}
\right]\in R~\text{and}~
\left[\begin{array}{cc}
0 & 1 \\
0 & 0 \\
\end{array}
\right]^2=\left[
            \begin{array}{cc}
              0 & 0 \\
              0 & 0 \\
            \end{array}
          \right]
.$$
\end{Ex}

\noindent Example~\ref{gggg} is a left morphic regular ring  which is not left strongly regular.  However, if $N$ is a reduced near-ring, then it fulfills the IFP-property.  But~\cite[Proposition 1.6.35]{ferrero2012nearrings} implies that  $N$ has the IFP-property if and only if for all $a\in N$,  $(0:_la)$ and $(0:_ra)$ are  ideals of $N$. Hence regularity conditions and left morphic left duo near-rings are linked by Proposition~\ref{226} below.

\begin{proposition}~\label{226}
Let $N$ be a  near-ring.  Then  the following  are equivalent:
\begin{enumerate}
\item[\emph{(1)}] $N$ is a reduced left morphic near-ring,
\item[\emph{(2)}] $N$ is a left strongly regular near-ring,
\item[\emph{(3)}] $N$ is a regular and left duo near-ring.
\end{enumerate}
\end{proposition}
\begin{proof}
$(1)\Rightarrow(2)$. Let $a\in N$.  By the left morphic hypothesis, there exists some $b\in N$ such that $Na=(0:_lb)$.   Further,  $N$ fulfills the IFP property because it is reduced.  Thus $(0:_l(0:_ra))=(0:_l(0:_rNa))=(0:_l(0:_r(0:_lb)))=(0:_lb)=Na$ so that  $(0:_l(0:_rNa))=Na$ for every $a\in N$.   By \cite[Proposition 3.4.14]{ferrero2012nearrings}, $N$ is regular  and therefore left strongly regular by Lemma~\ref{hdt}.

$(2)\Rightarrow(1)$. Given (2),  $N$ is left morphic by Proposition~\ref{ff} and reduced by  Lemma~\ref{hdt}.

$(3)\Rightarrow(2)$. Let $a\in N$.  Then each left annihilator set of $a$  is an ideal because $N$ is left duo.   In this case $N$ fulfills the IFP-property and so by~\cite[Proposition 4.2]{Groenewald}, $N$ is left strongly regular.

$(2)\Rightarrow(3)$. This always holds since every left strongly regular near-ring is regular by Lemma~\ref{hdt} and every $N$-ideal of  $N$ is an ideal.
\end{proof}

\begin{theorem}\label{62}
Every left morphic regular near-ring is unit-regular.
\end{theorem}
\begin{proof}
Let $a\in N$, $a=axa$ for some $x\in N$, $Na = (0:_lb)$ and $Nb=(0:_la)$ for some $b\in N$.     Observe that
$a\in Na = (0:_lb)~\text{ and}~ b\in Nb=(0:_la)$ give rise to
                               \begin{equation*}\label{eq10}
                                ab=0=ba.
                               \end{equation*}
The elements  $xa$ and $ax$ are idempotents and left-morphic; hence by  Theorem~\ref{217}, $$n(1-ax)=n-nax=-nax+n$$ for all $n\in N$.
From $a=axa$ we have   $(a-axa)=(1-ax)a=0$.
This implies that $(1-ax)\in (0:_la)=Nb$.
Therefore, $(1-ax)=rb$ for some $r\in N$.  Let  $u:=xax+b$.   We show that  $u$ is a unit.  In view of Proposition~\ref{64}, it
 is enough to show that $(0:_lu)=0$.   Suppose that we have a $y\in (0:_lu)$.  Then
\begin{equation}\label{eq11}
 0=yu=y(xax+b).
\end{equation}
Right multiply (\ref{eq11}) by $(1-ax)$ to get $0=y(xax+b)(1-ax)=y(xax(1-ax)+b(1-ax)).$ Applying Theorem~\ref{217} and the fact that $ab=0=ba$  leads to
$0=y((xax-xaxax) +b-bax)=yb.$  Hence  $y\in (0:_lb)=Na$.  That is, $y=na$ for some $n\in N$.
Then  $$0=0a=(yu)a=y(xax+b)a=y(xa+0)=yxa=naxa=na=y.$$
This proves that $(0:_lu)=0$ and thus  $u$ is a unit in $N$.     Lastly, we  observe that
$$aua=a(xax+b)a=a(xaxa+ba)=a(xa+0)=axa=a.$$  Hence $a$ is unit-regular with $u$ as the middle unit; and this clearly proves that $N$ is unit-regular.
\end{proof}

\noindent The near-ring  in  Example~\ref{cccxx} is left morphic and unit-regular near-ring.  Clearly, for rings,  all morphic regular rings are unit-regular and the converse holds by \cite{ehrlich1976units}.   In Examples~\ref{20} and~\ref{20c}, we construct  unit-regular near-ring elements and unit-regular near-rings which are not left morphic and hence the analogue of Ehrlich's theorem does not, in general, hold for near-rings.

\begin{Ex}~\label{20}
Consider the  group $\displaystyle{\mathbb{Z}_3:=\{\bar{0},\bar{1},\bar{2}\}}$ and the set of all functions
$$M_0(\mathbb{Z}_3):=\left\{f\in M_0(\mathbb{Z}_3)\mid f:\mathbb{Z}_3\rightarrow\mathbb{Z}_3~\text{such that}~f(\bar{0})=\bar{0}\right\}.$$
Recall, $M_0(\mathbb{Z}_3)$ is a zero-symmetric near-ring with respect to componentwise addition $+$ and composition $\circ$ of functions.

\begin{claim}
If $N:=M_0(\mathbb{Z}_3)$, then the near-ring $(N,+,\circ)$ is  unit-regular but not left morphic.
\end{claim}

\begin{proof}
We begin by investigating the elements of $N$ as follows.  Let $\bar{x}\in \mathbb{Z}_3,~f_i\in N$ with $i=1,\ldots, 9,$ then

\begin{center}
\begin{tabular}{ccc}
         $f_1(\bar{1})=\bar{0}$& and & $f_1(\bar{2})=\bar{0}$;   \\
         $f_2(\bar{1})=\bar{0}$& and & $f_2(\bar{2})=\bar{1}$;   \\
         $f_3(\bar{1})=\bar{0}$& and & $f_3(\bar{2})=\bar{2}$;   \\
         $f_4(\bar{1})=\bar{1}$& and & $f_4(\bar{2})=\bar{0}$;   \\
         $f_5(\bar{1})=\bar{1}$& and & $f_5(\bar{2})=\bar{1}$;   \\
         $f_6(\bar{1})=\bar{1}$& and & $f_6(\bar{2})=\bar{2}$;   \\
         $f_7(\bar{1})=\bar{2}$& and & $f_7(\bar{2})=\bar{0}$;   \\
         $f_8(\bar{1})=\bar{2}$& and & $f_8(\bar{2})=\bar{1}$;   \\
         $f_9(\bar{1})=\bar{2}$& and & $f_9(\bar{2})=\bar{2}$.   \\
   \end{tabular}
\end{center}

\noindent Note that $f_1=0_N$ is the zero map and $f_6=1_N$ is the identity map.  The only non-identity unity is $f_8$.  These three maps are trivially unit-regular.  Since $f_3,f_4,f_5$ and $f_9$ are idempotent, they are unit-regular and $f_2$ and $f_7$ are unit-regular using $f_8$.  Hence $N$ is a unit-regular near-ring.  But $N$ is not left morphic in view of Theorem~\ref{217} since $(f_5(1_N-f_5))(\overline{2})\neq \overline{0}$.
\end{proof}
\end{Ex}

\begin{Ex}~\label{20c}
Let  $D$  be a division ring    and  $R:=M_n(D)$ be the $n\times n$ matrix ring over $D,n\geq 2$.
$R$ is unit-regular \cite[Corollary 4.7]{goodearl1979neumann}.
For any non-zero $R$-module  $M$,
let $N: = R\times M$ and define addition and multiplication, respectively, in the following manner:
\begin{eqnarray*}
  \langle a_1, m_1\rangle + \langle a_2,m_2\rangle &=& \langle a_1 + a_2, m_1 + m_2\rangle \\
  \langle a_1, m_1\rangle\star\langle a_2, m_2\rangle &=& \langle a_1a_2, a_1m_2 + m_1\rangle.
\end{eqnarray*}

\begin{claim}
$(N,+,\star )$ is a unit-regular near-ring that is not left morphic.
\end{claim}
\begin{proof}
It can be verified that $N$ is a right near-ring with identity $\langle 1, 0\rangle$ which is not necessarily zero-symmetric.
We proceed to show that $N$ is unit-regular.   Let $\langle u, -um\rangle\in N$, where $u$ is a unit element in $R$ and $m\in M$.
Then for an arbitrary element $\langle a, m\rangle\in N, a\in R$, it is easy to verify that $\langle a, m\rangle\star \langle u, -um\rangle \star \langle a, m\rangle=\langle a, m\rangle$ and $\langle u, -um\rangle$ is a unit with inverse $\langle u^{-1}, m\rangle$.  This gives $N$  unit-regular.
To prove that $N$ is not left morphic
it suffices to give one element of $N$  that is not left-morphic.
Let $\langle a,m\rangle\in N,m\neq 0$.  If $\langle a,m\rangle$ is left morphic, then
\begin{equation*}
\langle a,m\rangle\in N\langle a,m\rangle= (0:_l\langle b,n\rangle)~\text{and}~\langle b,n\rangle\in N\langle b,n\rangle=(0:_l\langle a,m\rangle)
\end{equation*}
for some  $\langle b,n\rangle \in N$.   This gives $\langle a,m\rangle\star\langle b,n\rangle=\langle 0,0\rangle=\langle b,n\rangle\star\langle a,m\rangle$ which yields $ab=0=ba, -n=bm$ and $m=-an$.  Therefore, $m=-a(n)=-a(-bm)=(ab)m=0m=0$, which is a contradiction.  Thus $\langle a,m\rangle$ is not left morphic.
\end{proof}
\end{Ex}

\noindent By Proposition~\ref{ff} and Theorem \ref{62},  the picture of relationships among all these classes of near-rings is as follows:
$$
\left\{
\begin{array}{c}
                         {\rm Left~strongly} \\
                        {\rm regular~near\text{-}rings} \\
                     \end{array}
\right\}
\subsetneqq
\left\{\begin{array}{c}
                         {\rm Left~morphic} \\
                        {\rm regular~near\text{-}rings} \\
                       \end{array}
\right\}
\subsetneqq
\left\{\begin{array}{c}
                         {\rm Unit\text{-}} \\
                        {\rm regular~near\text{-}rings} \\
                       \end{array}
\right\}
.$$

\noindent However, if $N$ has IFP we get from~\cite[Proposition 4.2]{Groenewald} that such a unit-regular near-ring is left strongly regular.   Hence we have:

\begin{proposition}~\label{tttt}
Let $N$ be a near-ring which has IFP.  The following are equivalent:
\begin{enumerate}[(1)]
\item[\emph{(1)}] $N$ is left strongly regular.
\item[\emph{(2)}] $N$ is left morphic and regular.
\item[\emph{(3)}] $N$ is unit-regular.
\end{enumerate}
\end{proposition}

\begin{ack}
The authors are grateful to the referee for the helpful comments and suggestions that greatly improved the exposition.
This work was carried out at Makerere University with financial support from the Makerere-Sida Bilateral Program Phase IV, Project 316 ``Capacity building in Mathematics and its application''
\end{ack}

\end{document}